\pdfoutput=1
\RequirePackage{ifpdf}
\ifpdf 
\documentclass[pdftex]{sigma}
\else
\documentclass{sigma}
\fi

\numberwithin{equation}{section}

\newtheorem{Theorem}{Theorem}[section]
\newtheorem{Corollary}[Theorem]{Corollary}
\newtheorem{Proposition}[Theorem]{Proposition}

\begin{document}

\allowdisplaybreaks

\newcommand{\arXivNumber}{1808.03116}

\renewcommand{\PaperNumber}{021}

\FirstPageHeading

\ShortArticleName{Almost Lie Algebroids and Characteristic Classes}

\ArticleName{Almost Lie Algebroids and Characteristic Classes}

\Author{Marcela POPESCU and Paul POPESCU}
\AuthorNameForHeading{M.~Popescu and P.~Popescu}
\Address{University of Craiova, Faculty of Sciences, Department of Applied Mathematics,\\
13, ``Al. I. Cuza'' st., 200585 Craiova, Romania}
\Email{\href{mailto:marcelacpopescu@yahoo.com}{marcelacpopescu@yahoo.com}, \href{mailto:paul_p_popescu@yahoo.com}{paul\_p\_popescu@yahoo.com}}

\ArticleDates{Received August 10, 2018, in final form March 04, 2019; Published online March 20, 2019}

\Abstract{Almost Lie algebroids are generalizations of Lie algebroids, when the Jacobiator is not necessary null. A simple example is given, for which a Lie algebroid bracket or a~Courant bundle is not possible for the given anchor, but a natural extension of the bundle and the new anchor allows a Lie algebroid bracket. A cohomology and related characteristic classes of an almost Lie algebroid are also constructed. We prove that these characteristic classes are all pull-backs of the characteristic classes of the base space, as in the case of a~Lie algebroid.}


\Classification{53D17; 53C05; 58A99}

\section{Introduction}

The Lie algebroids (as, for example \cite{Mack2}) are generalizations of Lie algebras and integrable regular distributions. In fact a Lie algebroid is an anchored vector bundle (see \cite{PP-A02, PP, PP-A}) with a Lie bracket on the module of sections.

The use of Lie algebroids is very widespread in the literature. It is involved as a geometric settings, generalizing different geometric notions. But some relaxations of the definition of a~Lie algebroid are already studied under different names and different conditions (see~\cite{Gr01} for an up-to-day review).

The Jacobiator of the bracket has different forms, according to the setting (for example, for manifolds or supermanifolds, for skew symmetric or general brackets etc.) The Jacobi identity (i.e., the vanishing Jacobiator) is an essential condition for a Lie algebroid~\cite{DW}.

The use of relaxation conditions of Lie algebroids give a simple and elegant way to describe some mechanical systems. For example, linear almost Poisson structures~\cite{LMD}, quasi-Lie algebroids~\cite{GLMD}, or algebroids~\cite{GG, GJ}, but the list can be easily extended, most due to the same authors.

According also to \cite{LMD, GG, GLMD}, the interest in almost Lie algebroids comes also from that they are involved in nonholonomic geometry.

The aim of this paper is to investigate some properties of an almost Lie algebroid, i.e., an anchored vector bundle with an a corresponding almost Lie bracket. For sake of simplicity we consider the almost Lie case, but some
almost Lie conditions can be removed. Some abstract properties are given explicitly in some simple and concrete examples.

Specifically, we consider a trivial vector bundle $\pi _{E_{0}}\colon E_{0}\rightarrow {\mathbb R}^{2}$ having a four-dimensional fibre and we construct:
\begin{itemize}\itemsep=0pt
\item a specific anchor $\rho _{0}\colon E_{0}\rightarrow T{\mathbb R}^{2}$ (see formulas~(\ref{Anch01})) such that its image is surjective, except the origin, where its image is null and

\item a corresponding almost Lie bracket (see formula~(\ref{Brack01})) that gives an almost Lie algebroid structure on~$E_{0}$.
\end{itemize}

We prove (Theorem \ref{Th01-alg}) that there is no skew symmetric bracket to give a Lie algebroid structure on~$E_{0}$. It is interesting that if we remove the origin of the base, then the resulting vector bundle $E_{0}^{\prime }$ has a Lie algebroid bracket. But there is a subalgebroid $E_{0}^{\prime \prime }\subset E_{0}$ that is a~Lie algebroid that generate the same (generalized) distribution of $T{\mathbb R}^{2}$ as $E_{0}$, described above. Notice that the orbit set of the almost Lie algebroid $E_{0}$ has two elements: the origin (singular) and the rest of~${\mathbb R}^{2}$ (regular).

The derived anchored bundle constructed in~\cite{PP-A}, used for $E_{0}$, gives a Lie algebroid structure (Theorem~\ref{ThAlg02}). The proof of this result needs some long computations, based on some auxiliary and technical
results proved in Proposition~\ref{Prhelp01}.

The interest in $E_{0}$ supporting an almost Lie algebroid structure comes not only from the fact that it has no Lie algebroid structure, but also from the fact that it can not be a Courant vector bundle (Proposition~\ref{prCourant}). But according to the definitions related to the almost complex Lie algebroid case in~\cite{IP, PP-C}, $E_{0}$ has an almost complex endomorphism that has a null Nijenhuis tensor, i.e., it is integrable (Proposition~\ref{prComplex}).

Notice that the existence of a Lie algebroid bracket for a transitive Lie algebroid has obstructions (see, for example, \cite{Ku, Mack2,MiLe}).

The cohomology of a Lie algebroid was defined in \cite{MR} as the cohomology of an associated differential complex (see also~\cite{Mack2}). The characteristic classes of a Lie algebroid are constructed in~\cite{Fe01},
where there is proved that they are the pull back of the characteristic classes of the base manifold.

An interesting problem related to almost Lie algebroids comes to be a construction of a~cohomology theory and then some characteristic classes, related to this. The reasons come true as those in the Lie algebroid case.

In the last section we construct the cohomology complex of an almost Lie algebroid. In spite of the fact the complex is a quotient set, we use definitions adapted to corresponding sets. Using an analogous way as in \cite {Hu} to construct characteristic classes in the classical case, we construct the characteristic classes of an almost Lie algebroid. We prove an analogous result as in~\cite{Fe01}, i.e., the characteristic classes of a skew
algebroid are the pull back of the characteristic classes of the base manifold (Theorem~\ref{thChC}).

A construction of a cohomology theory for almost lie algebroids is given also in~\cite{GX}, using a~supergeometry setting. The calculations with cohomology classes in a such context are technically quite difficult, such
as to construct characteristic classes following a classical way. The main difficulty to handle classical tools in this case is that some objects can not be used related to some sections of vector bundles.

In our paper we consider a different approach. The cohomology classes are obtained in a more simply form, from a technical point of view. The differential complex elements are equivalence classes of antisymmetric forms of the vector bundle support of the almost Lie structure. For the sake of simplicity of our constructions (cohomology classes and then characteristic classes), we do not use the language of sheaf theory, nor any elaborated algebraic form of the differential complexes as equivalence classes.

Given the above, we think it is an interesting problem, but not a simple one, to relate the cohomology defined in the paper with that considered in~\cite{GX}.

\section{Skew symmetric algebroids}

Let $\pi _{E}\colon E\rightarrow M$ be a (smooth) vector bundle. An \emph{anchor} on $E$ is a vector bundle map $\rho \colon E\rightarrow TM$, where $\pi _{TM}\colon TM\rightarrow M$ is the tangent vector bundle of $M$; equivalently, an $\mathcal{F}(M)$-linear $\rho \colon \Gamma (E) \rightarrow \Gamma (TM) \overset{\rm not.}{=}{}\mathcal{X}(M) $. A \emph{skew symmetric bracket} on $E$ is a map $[\cdot ,\cdot ]_{E}\colon \Gamma
(E) \times \Gamma (E) \overset{\rm not.}{=}{}\Gamma (E) ^{2}\rightarrow \Gamma (E) $, such that $[Y,X]_{E}=-[X,Y]_{E}$, $[X,fY]_{E}=\rho (X) (f) Y+f[X,Y]_{E}$ and $\rho ( \lbrack X,Y]_{E}) =[\rho (X) ,\rho (Y) ]$, $\forall\, X,Y\in \Gamma (E) $ and $f\in \mathcal{F}(M)$. The tensor map $\mathcal{J}_{E}\colon \Gamma (E) ^{3}\rightarrow \Gamma (E) $, $\mathcal{J}_{E}(X,Y,Z) =\sum_{\rm circ}[X,[Y,Z]_{E}]_{E}$ is called the \emph{Jacobiator map} of the bracket. An almost Lie algebroid is a \emph{Lie algebroid} provided that the Jacobiator of the bracket vanishes.

If $\pi _{A}\colon A\rightarrow M$ is a vector bundle over $M$, then a \emph{linear} \emph{$E$-connection} on $A$ is a map $\nabla \colon \Gamma (E) \times \Gamma (A) \rightarrow \Gamma (A) $ that verify Koszul conditions: $\nabla _{fX}s=f\nabla _{X}s$, $\nabla _{(X+X^{\prime })}s=\nabla _{X}s+\nabla _{X^{\prime }}s$, $\nabla _{X}(fs) =\rho (X) (f) s+f\nabla _{X}s$, $\nabla_{X}( s+s^{\prime }) =\nabla _{X}s+\nabla _{X}s^{\prime }$, $\forall\, X,X^{\prime }\in \Gamma (E) $, $s,s^{\prime }\in \Gamma (A) $, $f\in \mathcal{F}(M)$. Its curvature is the tensor $R\colon \Gamma (E) ^{2}\times \Gamma (A) \rightarrow\Gamma (A) $, given by the formula
\begin{gather*}
R(X,Y) s=\nabla _{X}\nabla _{Y}s-\nabla _{Y}\nabla _{X}s-\nabla _{\lbrack X,Y]_{E}}s,
\end{gather*}
$\forall\, X,Y\in \Gamma (E) $ and $s\in \Gamma (A) $.

In particular, for an $E$-connection $\nabla $ on $E$ we can consider its torsion $T\colon \Gamma (E) ^{2}\rightarrow \Gamma (E) $, given by formula
\begin{gather*}
T(X,Y) =\nabla _{X}Y-\nabla _{Y}X-[X,Y]_{E}
\end{gather*}
and we denote its curvature
\begin{gather*}
R(X,Y) Z=\nabla _{X}\nabla _{Y}Z-\nabla _{Y}\nabla _{X}Z-\nabla_{\lbrack X,Y]_{E}}Z\overset{\rm not.}{=}{}\nabla _{X\wedge Y}Z,
\end{gather*}
$\forall\, X,Y,Z\in \Gamma (E)$.

The following result follows by straightforward computations.

\begin{Proposition}[first Bianchi identities] Let $\nabla $ be an $E$-connection on $E$ with torsion $T$ and curvature $R$. We have
\begin{gather*}
\sum_{\rm circ}R(X,Y) Z=\sum_{\rm circ} ( \nabla _{X}T )(Y,Z) +\sum_{\rm circ}T( T(X,Y) ,Z) +\mathcal{J}_{E}( X,Y,Z) .
\end{gather*}
\end{Proposition}

An immediate corollary is the following result.

\begin{Corollary}Consider an almost Lie algebroid $E\rightarrow M$.
\begin{enumerate}\itemsep=0pt
\item[$1.$] If there are locally $E$-connections on $E$\ with vanishing torsion and curvature, then $E$ is a~Lie algebroid, i.e., the bracket of $E$\ has a null Jacobiator. Or, equivalently:
\item[$2.$] If the bracket of $E$\ has a non-null Jacobiator on an open $U\subset M $, then there is no local $E$-connection on $E_{U}$\ with vanishing torsion and curvature.
\end{enumerate}
\end{Corollary}

Notice that the Jacobiator and the curvature are tensors only in the case when $E$ is an algebroid, thus the vanishing conditions above have sense only in this setting.

Let us consider the following anchored vector bundle $E_{0}\rightarrow M$ on the base $M={\mathbb R}^{2}$, where $E_{0}={\mathbb R}^{2}\times \mathcal{M}_{2} ({\mathbb R}) $ and the anchor defined as follows. In every point $\bar{x}=\big( x^{1},x^{2}\big) $:
\begin{alignat}{3}
& \rho \big( X_{1}^{1}\big) =\big( x^{1}\big) ^{2}\dfrac{\partial }{\partial x^{1}}, \qquad && \rho \big( X_{2}^{1}\big) =\big( x^{1}\big) ^{2}\dfrac{\partial }{\partial x^{2}}, & \nonumber\\
& \rho \big( X_{1}^{2}\big) =\big( x^{2}\big) ^{2}\dfrac{\partial }{\partial x^{1}},\qquad && \rho \big( X_{2}^{2}\big) =\big( x^{2}\big) ^{2}\dfrac{\partial }{\partial x^{2}}, & \label{Anch01}
\end{alignat}
where $X_{j}^{i}$ is the matrix with null entries, except $1$ on the position $(i,j)$, $\forall\, i,j\in \{ 1,2\} $. It is easy to see that the image by $\rho $ of sections of $E_{0}$ generates the whole tangent space $T_{\bar{x}}{\mathbb R}^{2}$ for $\bar{x}\neq \bar{0}$ and $\{\bar{0}\}\in T_{\bar{0}}{\mathbb R}^{2}$ for $\bar{x}= \bar{0}=(0,0)$. A section on $E_{0}$ is in $\ker \rho $ iff it is an $\mathcal{F}\big({\mathbb R}^{2}\big)$-combination of sections $\mathcal{X}_{1}=\big( x^{2}\big) ^{2}X_{1}^{1}-\big( x^{1}\big) ^{2}X_{1}^{2}$ and $\mathcal{X}_{2}=\big(x^{2}\big) ^{2}X_{2}^{1}-\big(x^{1}\big) ^{2}X_{2}^{2}$. We notice that these two sections do not generate a regular vector subbundle of~$E_{0}$. A general and more compact form, useful to extend in higher dimensions, is $\rho \big(X_{i}^{j}\big) =\big( x^{j}\big) ^{2}\frac{\partial }{\partial x^{j}}$.

In the sequel we consider some important examples using $E_{0}$.

There is a bracket that gives a skew algebroid structure on $E_{0}$, given
by
\begin{gather*}
\big[X_{1}^{1},X_{2}^{1}\big]_{E}=2x^{2}X_{2}^{1}, \qquad \big[X_{1}^{1},X_{1}^{2}\big]_{E_{0}}=-2x^{1}X_{1}^{2}, \qquad \big[X_{1}^{1},X_{2}^{2}\big]_{E_{0}}=0, \\
\big[X_{2}^{1},X_{1}^{2}\big]_{E_{0}}=2x^{2}X_{1}^{1}-2x^{1}X_{2}^{2}, \qquad \big[X_{2}^{1},X_{2}^{2}\big]_{E_{0}}=2x^{2}X_{2}^{1}, \qquad \big[X_{1}^{2},X_{2}^{2}\big]_{E_{0}}=-2x^{2}X_{1}^{2}.
\end{gather*} In a compact form
\begin{gather}
\big[ X_{j}^{i},X_{l}^{k}\big]_{E_{0}}=2x^{k}\delta_{j}^{k}X_{l}^{i}-2x^{i}\delta _{l}^{i}X_{j}^{k}. \label{Brack01}
\end{gather}
The Jacobiator is
\begin{gather*}
J\big( X_{1}^{1},X_{2}^{1},X_{2}^{2}\big) =J\big( X_{1}^{1},X_{1}^{2},X_{2}^{2}\big) =0, \qquad J_{2}F_{2}\big(X_{1}^{1},X_{2}^{1},X_{1}^{2}\big) =-2\mathcal{X}_{2},\\
J\big(X_{2}^{1},X_{1}^{2},X_{2}^{2}\big) =-2\mathcal{X}_{1}.
\end{gather*}

By a direct computation, we get
\begin{gather*}
\big[\mathcal{X}_{1},X_{1}^{1}\big]_{E_{0}}=\big[\mathcal{X}_{1},X_{1}^{2}\big]_{E_{0}}=\big[\mathcal{X}_{2},X_{2}^{1}\big]_{E_{0}}=\big[\mathcal{X}_{2},X_{2}^{2}\big]_{E_{0}}=0,\qquad
\big[\mathcal{X}_{1},X_{2}^{2}\big]_{E_{0}}=-2x^{2}\mathcal{X}_{1},\\
\big[\mathcal{X}_{1},X_{2}^{1}\big]_{E_{0}}=2x^{1}\mathcal{X}_{2}, \qquad \big[\mathcal{X}_{2},X_{1}^{1}\big]_{E_{0}}=-2x^{1}\mathcal{X}_{2}, \qquad \big[\mathcal{X}_{2},X_{1}^{2}\big]_{E_{0}}=2x^{2}\mathcal{X}_{1}.
\end{gather*}

Let us consider another bracket $[\cdot ,\cdot ]_{E_{0}}^{\prime }=[\cdot ,\cdot ]_{E_{0}}+B$, i.e., $[X,Y]_{E_{0}}^{\prime }=[X,Y]_{E_{0}}+B(X,Y)$, $\forall\, X,Y\in \Gamma (E_{0}) $, where $B\colon \Gamma (E_{0}) ^{2}\longrightarrow \Gamma (E_{0}) $ is an $\mathcal{F}\big({\mathbb R}^{2}\big)$-linear and sqew-symmetric map. If the bracket $[\cdot ,\cdot ]_{E_{0}}^{\prime }$ gives an algebroid, then $\rho (\lbrack X,Y]_{E_{0}}^{\prime }) =[\rho (X) ,\rho (Y) ]$. It follows that \smash{$\rho( B(X,Y))=0$}, thus $B(X,Y) =B^{1}(X,Y) \mathcal{X}_{1}+B^{2}(X,Y) \mathcal{X}_{2}$ defines two bilinear forms~$B^{1}$ and~$B^{2}$.

Let us denote by $J$ and $J^{\prime }$ the Jacobiators of the brackets $[\cdot ,\cdot ]_{E_{0}}$ and $[\cdot ,\cdot ]_{E_{0}}^{\prime }$ respectively. By a straightforward computations, we have $J^{\prime }(X,Y,Z) =J(X,Y,Z) +\mathcal{B}( X,Y,Z) $, where
\begin{gather*}
\mathcal{B}( X,Y,Z) =\sum\limits_{{\rm cicl.} \, X,Y,Z}\big([B(X,Y),Z]_{E_{0}}+B( [X,Y]_{E_{0}},Z) +B( B(X,Y),Z) \big).
\end{gather*}

Since $J\big( X_{1}^{1},X_{2}^{1},X_{2}^{2}\big) =J\big(X_{1}^{1},X_{1}^{2},X_{2}^{2}\big) =0$, there are induced Lie algebroids structures on the subbundles $E_{01}$ and $E_{02}$, generated by $\big\{X_{1}^{1},X_{2}^{1},X_{2}^{2}\big\}$ and by $\big\{X_{1}^{1},X_{1}^{2},X_{2}^{2}\big\}$ respectively.

Let us prove now that it is not possible to have
\begin{gather*}
B\big( X_{1}^{1},X_{2}^{1}\big) =B\big( X_{1}^{1},X_{1}^{2}\big) =B\big( X_{1}^{1},X_{2}^{2}\big) =B\big( X_{2}^{2},X_{2}^{1}\big) =B\big( X_{2}^{2},X_{1}^{2}\big) =0,
\end{gather*}
but $B\big( X_{2}^{1},X_{1}^{2}\big) =B_{1}\mathcal{X}_{1}+B_{2}\mathcal{X}_{2}$.

Indeed, by a direct computation, we get
\begin{gather*}
\mathcal{B}\big(X_{1}^{1},X_{2}^{1},X_{1}^{2}\big) =-\big( x^{1}\big) ^{2}B_{1,1}\mathcal{X}_{1}-\big( x^{1}\big) ^{2}B_{2,1}\mathcal{X}_{2}
\end{gather*} and
\begin{gather*}
\mathcal{B}\big( X_{2}^{1},X_{1}^{2},X_{2}^{2}\big) =-\big( x^{2}\big) ^{2}B_{1,2}\mathcal{X}_{1}-\big( x^{2}\big) ^{2}B_{2,2}\mathcal{X}_{2}.
\end{gather*}
Thus, in order that $J^{\prime }=0$, we must have $-\big( x^{1}\big) ^{2}B_{2,1}-2=0$, but this is not possible, for example for $x^{1}=0$. A~similar argument can be used to prove the following result.

\begin{Theorem}\label{Th01-alg}There is no Lie algebroid bracket corresponding to the anchor $\rho $.
\end{Theorem}

\begin{proof} Let us suppose that
\begin{gather*}
B\big( X_{1}^{1},X_{2}^{2}\big) =A_{1} \mathcal{X}_{1}+A_{2}\mathcal{X}_{2}, \qquad B\big( X_{1}^{1},X_{1}^{2}\big) =B_{1}\mathcal{X}_{1}+B_{2}\mathcal{X}_{2}, \\
B\big(X_{1}^{1},X_{2}^{1}\big) =C_{1}\mathcal{X}_{1}+C_{2}\mathcal{X}_{2}, \qquad B\big( X_{2}^{2},X_{2}^{1}\big) =D_{1}\mathcal{X}_{1}+D_{2}\mathcal{X}_{2} , \\
B\big( X_{2}^{2},X_{1}^{2}\big) =E_{1}\mathcal{X}_{1}+E_{2}\mathcal{X}_{2}, \qquad B\big( X_{2}^{1},X_{1}^{2}\big) =F_{1}\mathcal{X}_{1}+F_{2}\mathcal{X}_{2}.
\end{gather*}

By a direct computation, we get from
\begin{gather*}
J^{\prime }\big(X_{1}^{1},X_{2}^{e1},X_{1}^{2}\big) =J\big( X_{1}^{1},X_{2}^{1},X_{1}^{2}\big) +\mathcal{B}\big(X_{1}^{1},X_{2}^{1},X_{1}^{2}\big) =0,
\end{gather*} that there are smooth functions $f,g\colon {\mathbb R}^{2}\rightarrow {\mathbb R}$ such that $x^{1}f\big( x^{1},x^{2}\big) +x^{2}g\big( x^{1},x^{2}\big) -2=0$, but this is not possible, for
example for $x^{1}=0$ and $x^{2}=0$.
\end{proof}

\begin{Corollary}There is no $E_{0}$-linear connection on $E_{0}$ with simultaneously vanishing torsion and curvature.
\end{Corollary}

Consider now the following algebroid. The vector bundle $E_{0}^{\prime }$ over ${\mathbb R}^{2}$ is generated by four generators, denoted by $\big\{Y_{1}^{1},Y_{2}^{2},\mathcal{Y}_{1},\mathcal{Y}_{2}\big\}$, the anchor is
given by
\begin{gather*} \rho ^{\prime }\big( Y_{1}^{1}\big) =\big( x^{1}\big) ^{2}\frac{\partial }{\partial x^{1}}, \qquad \rho ^{\prime }\big( Y_{2}^{2}\big) =\big( x^{2}\big) ^{2}\frac{\partial }{\partial x^{2}},
\qquad \rho ^{\prime}\big( \mathcal{Y}_{1}\big) =\rho ^{\prime }\big( \mathcal{Y}_{1}\big) =0
\end{gather*} and the bracket of generators is
\begin{gather*}
\big[Y_{1}^{1},Y_{2}^{2}\big]_{E_{0}^{\prime}}=0, \qquad \big[Y_{1}^{1},\mathcal{Y}_{1}\big]_{E_{0}^{\prime }}= \big[Y_{2}^{2},\mathcal{Y}_{2}\big]_{E_{0}^{\prime }}=0, \qquad \big[Y_{1}^{1},\mathcal{Y}_{2}\big]_{E_{0}^{\prime}}= 2x^{1}\mathcal{Y}_{2},\\
\big[Y_{2}^{2},\mathcal{Y}_{1}\big]_{E_{0}^{\prime}}=2x^{2}\mathcal{Y}_{1}, \qquad \big[\mathcal{Y}_{1},\mathcal{Y}_{2}\big]_{E_{0}^{\prime}}=2\big( x^{1}\big) ^{2}x^{2}\mathcal{Y}_{1}+2x^{1}\big( x^{2}\big)^{2}\mathcal{Y}_{2}.
\end{gather*}

Let us consider the map
\begin{gather*}
f_{0}\colon \ E_{0}^{\prime }\rightarrow E_{0}, \qquad f_{0}\big( Y_{1}^{1}\big) =X_{1}^{1}, \qquad f_{0}\big( Y_{2}^{2}\big) =X_{2}^{2}, \qquad f_{0}\big( \mathcal{Y}_{1}\big) =\mathcal{X}_{1}, \qquad f_{0}(\mathcal{Y}_{2}) =\mathcal{X}_{2},
\end{gather*} then extended by generators. Then $f_{0}$ is a morphism of algebroids.

It is interesting that the new bracket given by
\begin{gather*}
\big[Y_{1}^{1},Y_{2}^{2}\big]_{E_{0}^{\prime }}^{\prime }=0, \qquad \big[Y_{1}^{1},\mathcal{Y}_{1}\big]_{E_{0}^{\prime }}^{\prime }=\big[Y_{2}^{2},\mathcal{Y}_{2}\big]_{E_{0}^{\prime}}^{\prime }=0, \qquad \big[Y_{1}^{1},\mathcal{Y}_{2}\big]_{E_{0}^{\prime }}^{\prime}=2x^{1}\mathcal{Y}_{2},\\
\big[Y_{2}^{2},\mathcal{Y}_{1}\big]_{E_{0}^{\prime}}^{\prime }=2x^{2}\mathcal{Y}_{1}, \qquad \big[\mathcal{Y}_{1},\mathcal{Y}_{2}\big]_{E_{0}^{\prime }}=0
\end{gather*} gives a Lie algebroid bracket on $E_{0}^{\prime }$. Obviously, there is not any Lie algebroid structure on~$E_{0}$ such that~$f_{0}$ be a Lie algebroid morphism.

In fact the vector subbundle $E_{0}^{\prime \prime }\subset E_{0}$ generated by $\mathcal{A}_{1}=X_{1}^{1}+X_{1}^{2}$ and $\mathcal{B}_{1}=X_{2}^{1}+X_{2}^{2}$ is a~subalgebroid and the induced bracket gives a~Lie algebroid structure on $E_{0}^{\prime \prime }$.

Denote now by $\mathcal{S}_{0}=\{\mathcal{A}_{1}^{\prime }, \mathcal{A}_{2}^{\prime }, \mathcal{X}_{1}^{\prime }, \mathcal{X}_{2}^{\prime }\}$ the restrictions of the corresponding sections $\mathcal{A}_{1}$, $\mathcal{A}_{2}$, $\mathcal{X}_{1}$, $\mathcal{X}_{2}$ to $E_{00}=E_{0}\backslash \{(0,0)\}\longrightarrow {\mathbb R}^{2}\backslash \{(0,0)\}$. Then $\mathcal{S}_{0}$ is a global base of sections $\Gamma ( E_{00})$. More,
considering the restriction of the anchor from $E_{0}$ to $E_{00}$ and the bracket $[\mathcal{X}_{1}^{\prime },\mathcal{X}_{2}^{\prime }]_{E_{00}}=[\mathcal{X}_{\alpha }^{\prime },\mathcal{A}_{\beta }^{\prime }]_{E_{00}}=0$,
$\alpha ,\beta =1,2$, and $[\mathcal{A}_{1}^{\prime },\mathcal{A}_{2}^{\prime }]_{E_{00}}=-2x^{2}\mathcal{A}_{1}^{\prime }+2x^{1}\mathcal{A}_{2}^{\prime }$, then one obtain a Lie algebroid structure on~$E_{00}$. Thus
removing the origin, the almost Lie algebroid on $E_{0}$ (that has not a Lie algebroid structure) changes essentially to the almost Lie algebroid on $E_{00}$, that allows a Lie algebroid structure. Notice that $\mathcal{S}_{0}$ can not extend to a base of~$\Gamma (E_{0})$.

As in the general case in \cite{PP-A}, we can consider the derived anchored bundle $E^{(1)}=E\oplus ( E\wedge E ) $, where the anchor is denoted as $\rho ^{(1)}$ and the bracket by $[\cdot ,\cdot ]_{E^{(1)}}$. Concretely, $\rho ^{(1)}(X) =\rho (X) $, $\rho ^{(1)}( X\wedge Y) =[\rho (X) ,\rho (Y)]-\rho ( \lbrack X,Y]_{E}) $, and consider an $E$-connection~$\nabla $ on $E$ that has a~null torsion. We can lift $\nabla $ to a linear $E^{(1)}$-connection $\nabla ^{(1)}$ on $E^{(1)}$ using formulas
\begin{gather*}\begin{split}
&\nabla _{X}^{(1)}Y=\nabla _{X}Y+\tfrac{1}{2}X\wedge Y, \qquad \nabla _{X\wedge Y}^{(1)}Z=\nabla _{X\wedge Y}Z, \\
&\nabla _{X}^{(1)}( Y\wedge Z)=\nabla _{X}( Y\wedge Z) =\nabla _{X}Y\wedge Z+X\wedge \nabla_{X}Z
\end{split}
\end{gather*} and
\begin{gather*} \nabla _{X\wedge Y}^{(1)}( Z\wedge T) =\nabla_{X\wedge Y}( Z\wedge T) =\nabla _{X\wedge Y}Z\wedge T+Z\wedge \nabla _{X\wedge Y}T.
\end{gather*}

We define the bracket $[\cdot ,\cdot ]_{E^{(1)}}$ on $E^{(1)}$ by $[U,V]_{E^{(1)}}=\nabla _{U}^{(1)}V-\nabla _{V}^{(1)}U$, i.e., such as $\nabla^{(1)}$ has a null torsion.

\begin{Proposition}\label{PrAlg}If $E$ is an algebroid, then $E^{(1)}$ is an algebroid iff
\begin{gather*}
\rho ( R(X,Y) Z) =0,
\end{gather*}
$\forall\, X,Y,Z\in \Gamma (E) $.
\end{Proposition}

\begin{proof} We have that $E$ is an algebroid iff $\rho ^{(1)}(X\wedge Y) =0$. Thus if $E^{(1)}$ is an algebroid, then
\begin{gather*}
\big[\rho^{(1)}( X\wedge Y) ,\rho ^{(1)}( Z) \big]=0=\rho^{(1)}\big( [X\wedge Y,Z]_{E^{(1)}}\big) \\
\qquad {} =\rho ^{(1)}( \nabla_{X\wedge Y}Z-\nabla z( X\wedge Y)) =\rho ^{(1)}(\nabla _{X\wedge Y}Z),
\end{gather*} and conversely, this equality ensures that~$E^{(1)}$ is an algebroid.
\end{proof}

We extend the definition of $\nabla _{X\wedge Y}Z=R(X,Y) Z$ to $\nabla _{X\wedge Y}( Z\wedge T) =\nabla _{X\wedge Y}Z\wedge T+Z\wedge \nabla _{X\wedge Y}T$.

The following properties of curvature $R^{(1)}$ of $\nabla ^{(1)}$ can be obtained by straightforward computations.

\begin{Proposition} \label{Prhelp01}If $X,X_{1},X_{2},Y,Y_{1},Y_{2},Z,T\in \Gamma (E) $, then
\begin{alignat*}{3}
& 1) \quad && R^{(1) }(X,Y) Z=R^{(1) }(X,Y)( Z\wedge T) =0,& \\
& 2) \quad && R^{(1)}( X_{1}\wedge X_{2},Y) Z=\nabla _{X_{1}\wedge X_{2}}\nabla _{Y}Z-\nabla _{Y}\nabla _{X_{1}\wedge X_{2}}Z-\nabla _{\nabla _{X_{1}\wedge X_{2}}Y}Z+\nabla _{\nabla _{Y}(X_{1}\wedge X_{2}) }Z,& \\
& 3) \quad && R^{(1)}( X_{1}\wedge X_{2},Y) ( Z\wedge T) =R^{(1)}( X_{1}\wedge X_{2},Y) Z\wedge T+Z\wedge R^{(1)}(X_{1}\wedge X_{2},Y) T, & \\
& 4) \quad && R^{(1)}( X_{1}\wedge X_{2},Y_{1}\wedge Y_{2}) Z=\nabla _{X_{1}\wedge X_{2}}\nabla _{Y_{1}\wedge Y_{2}}Z-\nabla_{Y_{1}\wedge Y_{2}}\nabla _{X_{1}\wedge X_{2}}Z-\nabla _{\nabla_{X_{1}\wedge X_{2}}( Y_{1}\wedge Y_{2}) }Z& \\
&&& \hphantom{R^{(1)}( X_{1}\wedge X_{2},Y_{1}\wedge Y_{2}) Z=}{} +\nabla _{\nabla_{Y_{1}\wedge Y_{2}}( X_{1}\wedge X_{2}) }Z,& \\
& 5) \quad && R^{(1)}( X_{1}\wedge X_{2},Y_{1}\wedge Y_{2}) (Z\wedge T) =R^{(1)}( X_{1}\wedge X_{2},Y) Z\wedge T+Z\wedge R^{(1)}(X_{1}\wedge X_{2},Y) T.&
\end{alignat*}
\end{Proposition}

We use now this constructions in the case of $E_{0}$.

An $E_{0}$-linear connection $\nabla $ on $E_{0}$ with vanishing torsion can be constructed by generators using the formulas
\begin{alignat*}{3}
& \nabla _{X_{1}^{1}}X_{2}^{1}=-\nabla _{X_{2}^{1}}X_{1}^{1}=x^{2}X_{2}^{1}, \qquad &&\nabla _{X_{1}^{1}}X_{1}^{2}=-\nabla _{X_{1}^{2}}X_{1}^{1}=-x^{1}X_{1}^{2},&\\
& \nabla _{X_{1}^{1}}X_{2}^{2}=-\nabla _{X_{2}^{2}}X_{1}^{1}=0, \qquad &&\nabla_{X_{2}^{1}}X_{1}^{2}=-\nabla_{X_{1}^{2}}X_{2}^{1}=x^{2}X_{1}^{1}-x^{1}X_{2}^{2}, &\\
& \nabla_{X_{2}^{1}}X_{2}^{2}=-\nabla _{X_{2}^{2}}X_{2}^{1}=x^{2}X_{2}^{1}, \qquad &&\nabla_{X_{1}^{2}}X_{2}^{2}=-\nabla _{X_{2}^{2}}X_{1}^{2}=-x^{2}X_{1}^{2}. &
\end{alignat*}
In a compact form $\nabla $ has the form
\begin{gather*}
\nabla _{X_{j}^{i}}X_{l}^{k}=2x^{k}\delta _{j}^{k}X_{l}^{i}
\end{gather*}

By a direct computation, we can get the following result.

\begin{Proposition}The curvature of the above $E_{0}$-connection $\nabla $ is
\begin{alignat*}{4}
& R\big(X_{1}^{1},X_{1}^{2}\big) X_{2}^{1}=-2\mathcal{X}_{2}, \qquad && R\big(X_{1}^{1},X_{1}^{2}\big) X_{1}^{1}=-2\mathcal{X}_{1},\qquad && R\big(X_{1}^{1},X_{1}^{2}\big) X_{2}^{1}=-2\mathcal{X}_{2},&\\
& R\big(X_{2}^{2},X_{2}^{1}\big) X_{1}^{2}=2\mathcal{X}_{1},\qquad && R\big(X_{2}^{2},X_{2}^{1}\big) X_{2}^{2}=2\mathcal{X}_{1}, \qquad && R\big(X_{2}^{2},X_{2}^{1}\big) X_{2}^{2}=2\mathcal{X}_{2}&
\end{alignat*} and the corresponding skew symmetric relations, but the other components are null.
\end{Proposition}

Since $\rho \left( R(X,Y) Z\right) =0$, using Proposition~\ref{PrAlg} we get

\begin{Corollary}Using the above $E_{0}$-connection $\nabla $, the derived anchored $E_{0}^{(1)}$ has an almost Lie algebroid structure.
\end{Corollary}

By a straightforward and a long computations, but using Proposition~\ref{Prhelp01} in order to shorten the calculations, we obtain

\begin{Theorem}\label{ThAlg02}The almost Lie algebroid $E_{0}^{(1)}$ is a Lie algebroid.
\end{Theorem}

According to \cite{Vas01} (see also \cite{BrGr, LiShXu02}), an anchored vector bundle $E$ with anchor $\rho $ is a \emph{Courant vector bundle} if there is a pseudo-euclidian metric~$g$ in the fibers of~$E$ such that $\rho
\circ \#\circ \rho ^{\ast }\colon T^{\ast }M\rightarrow TM$ vanishes, where $\#\colon E^{\ast }\rightarrow E$ is the musical isomorphism induced by~$g$.

\begin{Proposition}\label{prCourant}There is no pseudo-euclidean metric $g$ on $E_{0}$ such that $E_{0}$ is a Courant algebroid according to the given anchor.
\end{Proposition}

\begin{proof} Let us suppose that there is a such metric~$g$. Considering the $\mathcal{F}\big({\mathbb R}^{2}\big)$-generators $\big\{X_{1}^{1}, X_{2}^{1}, X_{1}^{2}, X_{2}^{2}\big\}$ of $\Gamma (E_{0}) $ and $\{\bar{e}_{1},\bar{e}_{2}\}$ of $\mathcal{X}\big({\mathbb R}^{2}\big)$, then the matrix of $\rho\colon E_{0}\longrightarrow T{\mathbb R}^{2}$ is $\big(\big( x^{1}\big) ^{2}I_{2} \; \big( x^{2}\big) ^{2}I_{2}\big)$ and the matrix of $g^{-1}$ is $\left(\begin{matrix} G_{1} & G_{2} \\
G_{2} & G_{3}
\end{matrix}\right)$, where $G_{1}$, $G_{2}$, $G_{3}$ are functions of variables $\big(x^{1},x^{2}\big)$. It follows that $\big( x^{1}\big) ^{4}G_{1}+2\big(x^{1}\big) ^{2}\big( x^{2}\big) ^{2}G_{2}+\big( x^{2}\big)
^{4}G_{3}=0_{2}$, where $0_{2}$ is the (two) square zero matrix. But considering $x^{2}=\lambda x^{1}$, $\lambda \neq 0$, we get that for $x^{1}\neq 0$, we have $G_{1}+2\lambda ^{2}G_{2}+\lambda ^{4}G_{3}=0$ in any
point $\big(x^{1},\lambda x^{1}\big)$, $x^{1}\neq 0$. For $x^{1}\rightarrow 0$, we obtain that the matrices $G_{i}^{0}=G_{i}(0,0)$, $i=1,2,3$, verify the equation $G_{1}^{0}+2\lambda ^{2}G_{2}^{0}+\lambda^{4}G_{3}^{0}=0$ for every $\lambda \neq 0$, thus $G_{i}^{0}=0_{2}$, $i=1,2,3$. But this is not possible for the metric~$g$.
\end{proof}

Analogously to the Lie algebroid case \cite{IP,PP-C}, an \emph{almost complex almost Lie algebroid} is a~real almost Lie algebroid $E$ such that there is an almost complex endomorphism on~$E$ (i.e., an endomorphism~$J_{E}$ of ${E}$ such that $J_{E}^{2}=-\mathrm{id}_{\Gamma(E)}$). The almost complex structure is \emph{integrable} if the Nijenhuis tensor $N_{J_{E}}\colon \Gamma (E) ^{2}\rightarrow \Gamma (E)$, $N_{J_{E}}(X,Y) = [ J_{E}X,J_{E}Y ] _{E}-J_{E}[ X,J_{E}Y] _{E}-J_{E}[ J_{E}X,Y] _{E}-[X,Y] _{E}$, vanishes, i.e., $N_{J_{E}}=0$.

On the almost Lie algebroid $E_{0}$ there is an almost complex endomorphism $J_{E_{0}}$ given by the formulas $J_{E_{0}}\big( X_{1}^{1}\big) =-X_{2}^{1}$, $J_{E_{0}}\big( X_{2}^{1}\big) =X_{1}^{1}$, $J_{E_{0}}\big(X_{1}^{2}\big) =-X_{2}^{2}$, $J_{E_{0}}\big(X_{2}^{2}\big) =X_{1}^{2}$. In a compact form, $J_{E_{0}}\big(X_{j}^{i}\big) =( -1) ^{i}X_{\tilde{j}}^{i}$, where $\tilde{j}=2$ if $j=1$ and $\tilde{j}=1$ if $j=2$.

\begin{Proposition}\label{prComplex}The above {almost complex structure }$J_{E_{0}}$ is integrable.
\end{Proposition}

\begin{proof} A straightforward computation leads to $N_{J_{E_{0}}}\big(X_{i}^{j},X_{l}^{k}\big) =0$.\end{proof}

Of course, since there is no Lie algebroid structure on~$E_{0}$, it can not be an almost complex (or even complex) Lie algebroid structure on~$E_{0}$.

\section{Characteristic classes of almost Lie algebroids}

We consider now an extension to almost Lie algebroids of the cohomology of Lie algebroids.

The background is also the exterior algebra of the dual $E^{\ast }$, denoted as $\Lambda ^{\ast }(E) $. The derivation~$d$ acts on functions $f\in \Lambda ^{0}(E) $ as $df(X) =\rho (X) (f) \overset{\rm not.}{=}{}[ X,f] _{E}$ (the differential of $f$) and for $\omega \in \Lambda ^{k}(E) $ its differential is
\begin{gather*}
d\omega ( X_{0},\ldots ,X_{k}) =\sum\limits_{i=0}^{k}(-1) ^{i}\big[ X_{0},\omega \big( X_{0},\ldots, \widehat{X_{i}},\ldots,X_{k}\big) \big] _{E}\\
\hphantom{d\omega ( X_{0},\ldots ,X_{k}) =}{} +\sum\limits_{0\leq i<j\leq k}( -1) ^{i+j}\omega \big( [ X_{i},X_{j}]_{E},X_{0},\ldots, \widehat{X_{i}},\ldots, \widehat{X_{j}},\ldots,X_{k}\big) .
\end{gather*}

We have that $d^{2}(f) =0$, but for $\omega \in \Lambda^{1}(E) $, $d^{2}( \omega) =\omega \circ J$, where $J $ is the Jacobiator of the bracket.

Considering a given almost Lie algebroid structure on $E$, let us denote as
\begin{itemize}\itemsep=0pt
\item $Z^{k}(E)\subset \Lambda ^{k}(E) $ the set of $\omega $ such that there is $\theta \in \Lambda ^{k-1}(E) $ such that $d\omega =d^{2}\theta $, we call $\omega $ a~\emph{strong closed} form;
\item $B^{k}(E)\subset Z^{k}(E) $ the set of $\omega \in Z^{k}(E) $ such that there is $\theta ^{\prime }\in \Lambda ^{k-1} (E) $ such that $\omega =d\theta ^{\prime }$, we call $\omega $ a~\emph{strong exact} form;
\item $Z_{0}^{k}(E)\subset \Lambda ^{k}(E) $ the set of $\omega $ such that $d\omega =0$ (it is easy to see that $Z_{0}^{k}(E)\subset Z^{k}(E) $);
\item $B_{0}^{k}(E)=B^{k}(E) \cap Z_{0}^{k}(E)$.
\end{itemize}

It is also easy to see that $Z_{0}^{\ast }(E) \subset \Lambda ^{\ast }(E) $ is an exterior subalgebra and $B_{0}^{\ast }(E) \subset Z_{0}^{\ast }(E) $ is an exterior ideal; also $B^{k}(E) \subset Z^{k}(E) $ and $B_{0}^{k}(E) \subset Z_{0}^{k}(E) $ are real vector subspaces. Consider, for every $k\geq 0,$ the real vector space quotients $H^{k}(E) =Z^{k}(E) /B^{k}(E) $ and $H_{0}^{k}(E) =Z_{0}^{k}(E) /B_{0}^{k}(E) $ and for $\omega \in Z^{k}(E) $ and $\omega ^{\prime }\in Z^{k}(E) $ denote their classes by $[ \omega ] $ and $[\omega ^{\prime }] _{0}$ respectively. Then $H_{0}^{\ast }(E) $ is an exterior real algebra. If $\omega \in Z^{k}(E)$ and $d\omega =d^{2}\theta $, then $\omega ^{\prime }=\omega -d\theta \in Z_{0}^{k}(E)$ and $[ \omega] =d\theta +[ \omega ^{\prime }] _{0}$. Notice that in the Lie algebroid case (when $d^{2}=0$), one has $[ \omega ] =[ \omega ^{\prime }] _{0}$, since $d\theta \in [ \omega ^{\prime }] _{0}$, more precisely $[d\theta] _{0}=[ 0] _{0}$; thus in this case the two types of cohomology classes coincides. In the general case, $[ d\theta]_{0}=[ 0] _{0}$ only when $[ d\theta] _{0}$ has sense, i.e., $d^{2}\theta =0$.

But we have to relax the above definitions, in order to construct higher order characte\-ris\-tic classes (of order at least four). Consider a domain $U\subset M$ of trivialisation and a base $\{ s_{\alpha }\}_{\alpha =1,\dots,k}\subset \Gamma ( E_{U}) $, as well the dual base $\{ \omega ^{\alpha }\} _{\alpha =1,\dots,k}\subset \Gamma ( E_{U}^{\ast }) $. We denote by $\Lambda^{(2)\ast }( E_{U}) $ the exterior ideal generated in $\Lambda^{\ast }( E_{U}) $ by the set $\big\{ d^{2}\omega ^{\alpha}\big\} _{\alpha =1,\dots,k}\subset \Lambda ^{2}( E_{U}) $. Considering an open cover of $M$ with such $U$, For $k\geq 3$, denote by $D^{k}(E) $ the set of differential forms $\omega \in \Lambda^{k}(E) $ such that $\omega _{U}\in \Lambda ^{(2)k}(E_{U}) $, where $\omega _{U}$ denotes the restriction. For $k\in \{0,1,2\}$, $D^{k}(E) =\{ 0\} $. We have that $\Lambda^{(2)\ast }(E) \subset \Lambda ^{\ast }(E) $ is an exterior ideal and $d$ induces a differen\-tial~$\bar{d}$ on $\bar{\Lambda}^{\ast }(E) =\Lambda ^{\ast }(E) /\Lambda ^{(2)\ast}(E) $. It is easy to see that $\bar{d}^{2}=0$ and $\bar{\Lambda}^{\ast }(E) $ becomes a differential algebra. There is a~morphism of differential algebras $\rho ^{\ast }\colon \Lambda (M)\rightarrow \bar{\Lambda}^{\ast }(E) $, induced by the anchor $\rho \colon E\rightarrow TM$.

Accordingly, we denote as
\begin{itemize}\itemsep=0pt
\item $\mathcal{Z}^{k}(E)\subset \Lambda ^{k}(E) $ the set of $\omega \in \Lambda ^{k}(E)$ such that $d\omega \in \Lambda ^{(2)k+1}(E) $, we call $\omega $ a~\emph{weak closed} form;

\item $\mathcal{B}^{k}(E)\subset \mathcal{Z}^{k}(E) $ the set of $\omega \in \mathcal{Z}^{k}(E)$ such that there are $\theta ^{\prime }\in \Lambda ^{k-1}(E) $ and $\omega ^{\prime }\in \Lambda^{(2)k}(E) $ such that $\omega =\omega ^{\prime }+d\theta^{\prime }$, we call $\omega $ a~\emph{weak exact} form.
\end{itemize}

It is easy to see that $\mathcal{Z}^{k}(E)\subset Z^{k}(E)$ and $\mathcal{B}^{k}(E)\subset B^{k}(E)$, i.e., a form that is strong closed (strong exact) is also weak closed (weak exact respectively).

For every $k\geq 0$, $\mathcal{B}^{k}(E) \subset \mathcal{Z}^{k}(E) $ is a real vector subspace; we call the quotient real vector space $\mathcal{H}^{k}(E) =\mathcal{Z}^{k}(E)/\mathcal{B}^{k}(E) $ as the \emph{order }$k$ \emph{real cohomology} of~$E$. Then $\mathcal{H}^{\ast }(E) $ is an exterior real algebra, that we call the \emph{weak real cohomology} of~$E$.

In the Lie algebroid case, one has $\Lambda ^{(2)k+1}(E) =\{ 0\} $, $\mathcal{Z}^{k}(E) =Z^{k}(E) $, $\mathcal{B}^{k}(E) =B^{k}(E) $ and $\mathcal{H}^{k}(E) =H^{k}(E) $ is the $k$-cohomology space of the Lie algebroid (as, for example in \cite{Vas01}).

First, let us see what can be recovered from Cartan calculus, in the case of an almost Lie algebroid.

Let $\pi _{A}\colon A\rightarrow M$ be a vector bundle over $M$. As in the case of Lie algebroids (as in~\cite{Fe01}), a~linear $E$-connection on~$A$ is a~map $\nabla \colon \Gamma (E) \times \Gamma (A) \rightarrow
\Gamma (A) $ that fulfills Koszul conditions, i.e., it is $\mathcal{F}(M)$-linear in the first argument and, considering the second argument, it is a~$1$-derivation according to $\mathcal{F}(M)$ (see~\cite{PP-A}). The $E$-curvature of $\nabla $ is the $\mathcal{F}(M)$-linear map $R\colon \Gamma (E) \times \Gamma (E) \times \Gamma (A) \rightarrow \Gamma (A) $, $R(X,Y) B=\nabla_{X}\nabla _{Y}B-\nabla _{Y}\nabla _{X}B-\nabla _{[ X,Y]_{A}}B$.
Consider now bases of local sections over the same open subsets $U\subset M$, $\{ s_{\alpha }\} _{\alpha =1,\dots,k}\subset \Gamma( E_{U}) $ and $\{ \bar{s}_{a}\} _{a=1,\dots,n}\subset \Gamma (A_{U})$. Let us consider the dual bases $\{ \omega ^{\alpha }\} _{\alpha =1,\dots,k}\subset \Gamma( E_{U}^{\ast }) $ and $\{ \bar{\omega}^{a}\} _{a=1,\dots,n}\subset \Gamma ( A_{U}^{\ast }) $. A linear $E$-connection on~$A$ gives the local functions $\Gamma _{\beta b}^{a}=\bar{\omega}^{a}( \nabla _{s_{\beta }}\bar{s}_{b})$, $R_{\alpha \beta b}^{a}=\bar{\omega}^{a}( R(s_{\alpha },s_{\beta })\bar{s}_{b}) $ and the local forms $\theta _{b}^{a}=\Gamma _{\beta b}^{a}\omega ^{\beta}\in \Gamma ( E_{U}^{\ast }) $ and $R_{b}^{a}=\frac{1}{2}R_{\alpha \beta b}^{a}\omega ^{\alpha }\wedge \omega ^{\beta }$. The following Cartan formula follows by a straightforward verification.

\begin{Proposition}[\cite{Vas01}]The following formulas holds
\begin{gather}
R_{b}^{a}=d\theta _{b}^{a}+\sum_{c=1}^{n}\theta _{c}^{a}\wedge \theta _{b}^{c} . \label{eqCartan01}
\end{gather}
\end{Proposition}

If one denotes the matrices $\theta =\big( \theta _{b}^{a}\big) _{a,b=1,\dots,n}$ (of $1$-forms) and $\bar{R}=\big( R_{b}^{a}\big) _{a,b=1,\dots,n}$ (of $2$-forms), the above formula~(\ref{eqCartan01}) has the form
\begin{gather}
\bar{R}=d\theta +\theta \wedge \theta . \label{eqCartan02}
\end{gather}

Differentiating both sides of formula (\ref{eqCartan02}), then using the same formula for $d\theta $, we obtain
\begin{gather}
d\bar{R}=d^{2}\theta +\bar{R}\wedge \theta -\theta \wedge \bar{R}.\label{eqdR}
\end{gather}
Considering the traces $\bar{R}_{2}=\operatorname{Tr}\bar{R}\in \Lambda ^{2}(E) $ and $\theta _{0}=\operatorname{Tr}\theta \in \Lambda ^{1}( E_{U}) $, then noticing that $\operatorname{Tr}( \bar{R}\wedge \theta -\theta \wedge \bar{R}) =\operatorname{Tr}( \bar{R}\wedge \theta ) -\operatorname{Tr}( \theta \wedge \bar{R}) =0$, it follows that
\begin{gather*}
d\bar{R}_{2}=d^{2}\theta _{0},
\end{gather*}
thus $\bar{R}_{2}\in Z^{2}(E) \subset \mathcal{Z}^{2}(E) $, i.e., it is a closed $2$-form and $d^{2}\theta _{0}\in \Lambda^{2}(E) $ is a global $2$-form. We follow now a classical way for constructing characteristic classes (for example~\cite{Hu}),

If $\nabla $ is a metric linear connection, then $\theta _{0}=0$, thus $d\bar{R}_{0}=0$, i.e., $\bar{R}_{0}\in Z_{0}^{2}(E) $.

For $k\geq 1$, the $k$-order characteristic class is defined by the $2k$-form $\bar{R}_{2k}=\operatorname{Tr} \bar{R}^{k}\in \Lambda ^{2k}(E) $, where $\bar{R}^{k}$ is $\bar{R}\wedge \cdots \wedge \bar{R}$ ($k$ times).

\begin{Proposition}The form $\bar{R}_{2}\in \mathcal{Z}^{2}(E)$, i.e., it is strong closed, and for $k\geq 2$, $\bar{R}_{2k}\in \mathcal{Z}^{2k}(E)$, i.e., it is weak closed.
\end{Proposition}

\begin{proof} The proof for $k=1$ is performed above, and for $k\geq 2$ we follow the same idea used in~\cite{Hu}.

We have, for $k\geq 2$:
\begin{gather*} d\operatorname{Tr}\bar{R}^{k}=\operatorname{Tr}d\bar{R}^{k}=\sum\limits_{i+j=k-1}\operatorname{Tr}\big(\bar{R}^{i}\wedge d\bar{R}\wedge \bar{R}^{j}\big)\\
\hphantom{d\operatorname{Tr}\bar{R}^{k}}{}
=\sum\limits_{i+j=k-1}\operatorname{Tr}\big(\bar{R}^{i}\wedge \big( d^{2}\theta +\bar{R}\wedge \theta -\theta \wedge \bar{R}\big) \wedge \bar{R}^{j}\big)\\
\hphantom{d\operatorname{Tr}\bar{R}^{k}}{}
=k\operatorname{Tr}\big(d^{2}\theta \wedge \bar{R}^{k-1}\big)+k\operatorname{Tr}\big(\bar{R}^{k}\wedge \theta -\theta \wedge \bar{R}^{k}\big)=k\operatorname{Tr}\big(d^{2}\theta \wedge \bar{R}^{k-1}\big),
\end{gather*}
 since
\begin{gather*} \operatorname{Tr}\big(\bar{R}^{k}\wedge \theta -\theta \wedge \bar{R}^{k}\big)=\operatorname{Tr}\big(\bar{R}^{k}\wedge \theta \big)-\operatorname{Tr}\big(\theta \wedge \bar{R}^{k}\big)=0.
\end{gather*} Thus the conclusion follows.
\end{proof}

We prove now that the cohomology class of $\bar{R}_{2k}$ is the same for different $E$-linear connections on $A$. In order to prove this, we consider the skew algebroid product $\tilde{E}=E\times T{\mathbb R}$, as in the classical way for Lie algebroids~\cite{Vas01} and for differentiable manifolds in \cite{Hu} and the induced vector bundle $\tilde{A}$ over $M\times {\mathbb R}$. Every two $E$-linear connections $\nabla ^{(1)}$ and $\nabla ^{(2)}$ on $A$ gives rise to a~$\tilde{E}$-linear connection $\tilde{\nabla}$ on $\tilde{A}$, by $\tilde{\nabla}_{X}s=(1-t) \nabla _{X}^{(1)}s+t\nabla _{X}^{(2)}s$ and $\tilde{\nabla}_{\frac{\partial }{\partial t}}s=0$, $\forall\, X\in \Gamma (E)$ and $s\in \Gamma (A) $. Considering the morphisms of almost Lie algebroids $\tilde{I}_{u}=(i_{u},I_{u})$, where $i_{u}\colon M\rightarrow M\times {\mathbb R}$ and $I_{u}\colon E\rightarrow \tilde{E}$, $u=0,1$, $i_{u} ( x ) =(x,u)$, $I_{u}(e)=(e,u)$. The vector bundles $i_{u}^{\ast }E$ are canonically isomorphic with~$E$ and the restrictions of $\tilde{\nabla}$ to $i_{u}^{\ast }E$ coincides with~$\nabla $. More, concerning the curvatures $R^{(u)}$ of the two linear $E$-connections and the curvature~$\tilde{R}$, the restrictions~$\tilde{I}_{u}^{\ast }$ of the curvature tensors are $\tilde{I}_{u}^{\ast }\tilde{R}=R^{(u)}$, $u=0,1$. Thus $\tilde{I}_{u}^{\ast }\tilde{R}^{k}=\big( R^{(u)}\big) ^{k}$.

We consider now an integration operator $H\colon \Lambda ^{p+1}\big( \tilde{E}\big) \rightarrow \Lambda ^{p}(E) $ of the differential forms of~$\tilde{E}$, on the real fibers of~$\tilde{E}\rightarrow E$, where $p\geq 0$. More exactly:
\begin{itemize}\itemsep=0pt
\item[--] if a local $p+1$ form $\tilde{\omega}$ on $\tilde{E}$ is locally generated by forms induced from the fibers of $E$, i.e., it is a sum of local forms $\tilde{\theta}_{(x,t)}=f(x,t)\theta _{x}$, then $H(\tilde{\omega}) =0$;
\item[--] if a local $p+1$ form $\tilde{\omega}$ on $\tilde{E}$ is \emph{not} locally generated by forms induced from the fibers of~$E$, i.e., it is a sum of local forms $\tilde{\theta}_{(x,t)}=f(x,t)\theta _{x}\wedge dt$, then $H( \tilde{\omega}) =\left( \int_{0}^{1}f(x,t) \right)\theta _{x}$.
\end{itemize}

Notice that the local considerations correspond to domains of local trivial decomposition charts of $E$, and the definition of $H$ does not depend on local decompositions or coordinates, i.e., it is a global one.

By a straightforward computation one can prove the following result.

\begin{Proposition}
The global operator $H$ has the following properties:
\begin{enumerate}\itemsep=0pt
\item[$1)$] $H\circ \tilde{d}+d\circ H=\tilde{I}_{1}^{\ast }-\tilde{I}_{0}^{\ast }$;
\item[$2)$] $H\circ \tilde{d}^{2}=d^{2}\circ H$;
\item[$3)$] $H\big( \Lambda ^{(2)k+1} \big( \tilde{E} \big) \big) \subset\Lambda ^{(2)k} \big( \tilde{E} \big) $.
\end{enumerate}
\end{Proposition}

As a consequence, we can prove the following result, that allow to construct the characteristic classes of almost Lie algebroids.

\begin{Theorem}\label{thChC}If $\nabla $ is an $E$-linear connection on $A$, then each weak cohomology class $\big[ R^{k}\big] $, $k\geq 1$ has the following properties:
\begin{enumerate}\itemsep=0pt
\item[$1)$] it is independent of $\nabla $;
\item[$2)$] it defines a $2k$-cohomology class of $\big( \bar{\Lambda}^{\ast}(E) ,\bar{d}\big) $, induced by $\rho ^{\ast}\colon \Lambda ^{\ast}(M) \rightarrow \bar{\Lambda}^{\ast }(E) $ from a~$2k$-characteristic class on the base~$M$.
\end{enumerate}
\end{Theorem}

\begin{proof} It remains to prove only that there is an $E$-connection $\nabla $ on $A$ that is induced by the bracket using a linear connection $D$ on $A$. Indeed, it is $\nabla _{X}s=D_{\rho (X) }s$. This ends the proof.
\end{proof}

Let us explicit some facts concerning the particular almost Lie algebroid $E_{0}$, constructed previously.

Let us denote by $\mathcal{F}_{0}\subset \mathcal{F}\big({\mathbb R}^{2}\big)$ the set of global (smooth) functions $f\colon {\mathbb R}^{2}\rightarrow {\mathbb R}$ having the form $f\big( x^{1},x^{2}\big) =\big( x^{1}\big)^{2}f_{1}\big(x^{1},x^{2}\big) +\big( x^{2}\big) ^{2}f_{2}\big( x^{1},x^{2}\big) $, where $f_{1}, f_{2}\in \mathcal{F}\big({\mathbb R}^{2}\big)$. Let us denote by $\big\{ \omega _{j}^{i}\big\} _{i,j=1,2}\subset \Gamma (E_{0}^{\ast }) $ the dual base of $\big\{ X_{j}^{i}\big\} _{i,j=1,2}\subset \Gamma (E_{0}) $. Considering the derivation~$d$ on~$\Lambda ^{\ast }(E_{0}) $, a straightforward computation leads to

\begin{itemize}\itemsep=0pt
\item $d^{2}\big( \Lambda ^{0}(E_{0}) \big) \subset \Lambda ^{2}(E_{0}) $ is the null set;
\item $d^{2}\big( \Lambda ^{1}(E_{0}) \big) \subset \Lambda ^{3}(E_{0}) $ is the set of forms $\omega $ of degree three having the form $\omega =f_{1}\omega _{2}^{1}\wedge \omega _{1}^{2}\wedge \omega _{2}^{2}+f_{2}\omega _{1}^{1}\wedge \omega _{2}^{1}\wedge \omega_{1}^{2}$, where $f_{1}, f_{2}\in \mathcal{F}_{0}$;
\item $d^{2}\big( \Lambda ^{2}(E_{0}) \big) \subset \Lambda ^{2}(E_{0}) $ is the set of forms $\omega $ of degree four having the form $\omega =f\omega _{1}^{1}\wedge \omega _{2}^{1}\wedge \omega _{1}^{2}\wedge \omega _{2}^{2}$, where $f\in \mathcal{F}_{0}$;
\item $d^{2}\big( \Lambda ^{k}(E_{0}) \big) \subset \Lambda ^{k+2}(E_{0}) $ is the null set, for $k>2$.
\end{itemize}

Since on $E_{0}$ there is a base of the $\mathcal{F}\big({\mathbb R}^{2}\big)$-module of sections, i.e., $\big\{ X_{j}^{i}\big\} _{i,j=1,2}\subset \Gamma ( E_{0}) $, and~$E_{0}$ is an almost Lie algebroid, then there is an $E_{0}$-connection on $E_{0}$ (extending the conditions $\nabla_{X_{j}^{i}}X_{l}^{k}=0$, by Koszul's rules) having a null curvature, thus its characteristic classes vanish in all dimensions.

\subsection*{Acknowledgements}
The authors thank all three distinct referees for their valuable comments that helped us to improve the content of the paper. The research was supported by Horizon2020-2017-RISE-777911 project.

\pdfbookmark[1]{References}{ref}
\LastPageEnding

\end{document}